\documentclass[10pt]{article}

\usepackage{amsmath,amssymb,amsthm,mathrsfs}

\makeatletter
\renewenvironment{proof}[1][\proofname]{\par
  \pushQED{\qed}%
  \normalfont\topsep6\p@\@plus6\p@\relax
  \trivlist
  \item[\hskip\labelsep\bfseries #1\@addpunct{.}]\ignorespaces
}{%
  \popQED\endtrivlist\@endpefalse
}
\makeatother
\usepackage{enumerate}
\usepackage[margin=3cm]{geometry}
\usepackage{hyperref}
\hypersetup{hidelinks}
\newtheorem{theo}{Theorem}[section]
\newtheorem{lem}[theo]{Lemma}

\newtheorem{prop}[theo]{Proposition}
\newtheorem{cor}[theo]{Corollary}

\numberwithin{equation}{section}
\allowdisplaybreaks

\newcommand\ep{{\varepsilon}}
\newcommand{\R}{\mathbb{R}}

\newcommand{\Lip}{\operatorname{Lip}}
\newcommand{\diver}{\operatorname{div}}

\begin{document}

\title{Inviscid limit for the hyperbolic Keller--Segel equations in $B_{\infty,1}^{1}$}
\author{Yang Cen\textsuperscript{1}\thanks{E-mail: ceny3@mail2.sysu.edu.cn},\quad
Zhaoyang Yin\textsuperscript{1}\thanks{E-mail: mcsyzy@mail.sysu.edu.cn}\\
\textsuperscript{1}School of Science,\\
Shenzhen Campus of Sun Yat-sen University, Shenzhen 518107, China}
\date{}
\maketitle

\begin{abstract}
In this paper, we study the inviscid limit for the hyperbolic
Keller--Segel equations with logistic sensitivity $1-u$.
We establish a common existence time and a priori estimates uniform
with respect to the diffusion coefficient in the critical Besov spaces
$B_{p,1}^{1+d/p}(\R^d)$ with $1\le p\le\infty$.
For $1<p<\infty$, we obtain explicit convergence rates in $B_{p,1}^{0}$,
and we also give estimates for the case $p=1$.
For initial data in $B_{\infty,1}^{1}$, we prove convergence at rate
$O(\ep^{1/2})$ in $L^\infty$, a bound of order
$O(\ep^{1/2}\log(e/\ep))$ in $B_{\infty,1}^{0}$, and strong convergence
in $L^\infty(0,T;B_{\infty,1}^{1})$.
Additional regularity of the initial data yields improved convergence
rates. We construct a fixed initial datum showing that the exponent
$1/2$ in $L^\infty$ is optimal.

\end{abstract}

\section{Introduction}

Chemotaxis is the directed movement of cells or organisms in response
to chemical gradients. Patlak \cite{Patlak1953} introduced a biased
random-walk description of this process. Keller and Segel
\cite{KellerSegel1970,KellerSegel1971a,KellerSegel1971b} subsequently
proposed continuum models for slime mold aggregation and bacterial
traveling bands. For an overview of chemotaxis models and their
mathematical analysis, we refer to
\cite{Horstmann2003,HillenPainter2009}.
A general chemotaxis system takes the form
\[
\begin{cases}
\partial_t u=\nabla\cdot\bigl(D_u(u,S)\nabla u-\chi(u,S)u\nabla S\bigr)+H(u,S),\\
\tau\partial_t S=D_S(u,S)\Delta S+K(u,S),
\end{cases}
\]
where $u$ is the cell density and $S$ is the concentration of the
chemical signal. The coefficients $D_u$ and $D_S$ describe diffusion,
$\chi$ is the chemotactic sensitivity, and $H$ and $K$ are source terms.
The parameter $\tau\ge0$ describes the relaxation of the chemical signal.

In this paper, we consider the logistic sensitivity $\chi(u,S)=1-u$,
which accounts for the volume-filling effect, and the elliptic relation
$(I-\Delta)S=u$. The corresponding parabolic system is
\begin{equation}\label{eq:PKS}
\begin{cases}
\partial_t u^\ep-\ep\Delta u^\ep+\diver\bigl(u^\ep(1-u^\ep)\nabla S^\ep\bigr)=0,\\
(I-\Delta)S^\ep=u^\ep,\\
u^\ep|_{t=0}=u_0,
\end{cases}
\end{equation}
where $0<\ep\le1$ is the diffusion coefficient and
$(t,x)\in[0,T]\times\R^d$, $d\ge1$.
Setting $\ep=0$ gives the hyperbolic Keller--Segel system
\begin{equation}\label{eq:HKSE}
\begin{cases}
\partial_t u+\diver\bigl(u(1-u)\nabla S\bigr)=0,\\
(I-\Delta)S=u,\\
u|_{t=0}=u_0.
\end{cases}
\end{equation}
Throughout the paper, the initial data and solutions take values in $\mathbb{R}$.
Expanding the flux and using $S=(I-\Delta)^{-1}u$, we can rewrite
the systems \eqref{eq:PKS} and \eqref{eq:HKSE} in the unified form
\begin{equation}\label{eq:c5-unified}
\partial_tu+(1-2u)\nabla (I-\Delta)^{-1}u\cdot\nabla u-\ep\Delta u
=-u(1-u)((I-\Delta)^{-1}u-u),\qquad \ep\in[0,1].
\end{equation}

There is an extensive literature on the qualitative properties of
solutions to the Keller--Segel equations. J\"ager and Luckhaus
\cite{JagerLuckhaus1992} studied finite-time concentration.
Nagai, Senba and Yoshida \cite{NagaiSenbaYoshida1997} used the
Moser--Trudinger inequality to analyze parabolic chemotaxis systems.
Biler \cite{Biler1998} investigated local and global solvability.
Corrias, Perthame and Zaag \cite{CorriasPerthameZaag2004} obtained
global weak solutions for related chemotaxis and angiogenesis systems in
high dimensions. Blanchet, Dolbeault and Perthame
\cite{BlanchetDolbeaultPerthame2006} studied the critical mass mechanism
for the two-dimensional parabolic--elliptic model. Further results on
finite-time blow-up and the global behavior of solutions were obtained
by Winkler \cite{Winkler2010,Winkler2013}.

For models incorporating the volume-filling effect, Hillen and Painter
\cite{HillenPainter2001} established global existence for a parabolic
chemotaxis system with prevention of overcrowding. Dolak and Schmeiser \cite{DolakSchmeiser2005} studied the
small-diffusivity limit of the one-dimensional model with
logistic sensitivity on a bounded interval. Burger, Di Francesco and Dolak-Struss
\cite{BurgerDiFrancescoDolakStruss2006} investigated weak solutions and
large-time behavior. Burger, Dolak-Struss and Schmeiser
\cite{BurgerDolakStrussSchmeiser2008} analyzed the advection-dominated
limit in several dimensions. Perthame and Dalibard \cite{PerthameDalibard2009} constructed
entropy solutions of the hyperbolic model on bounded domains
by vanishing viscosity. For related hyperbolic chemotaxis models, Derbel and Jabin
\cite{DerbelJabin2007} studied concentration sets.
Lee and Liu \cite{LeeLiu2015} obtained a threshold criterion
for shock formation in one dimension.

In the framework of Besov spaces, Zhou, Zhang and Mu
\cite{ZhouZhangMu2021} proved local well-posedness and nonuniform
dependence on the initial data in $B_{p,r}^{s}$ for $s>1+d/p$.
Zhang, Mu and Zhou \cite{ZhangMuZhou2022} established local well-posedness
in $B_{p,1}^{1+d/p}$ for $1\le p<\infty$, studied the inviscid limit in
the same Besov topology as the initial data, and derived blow-up
criteria. For a Keller--Segel model with nonlinear chemotactic sensitivity
and signal secretion, Zhou and Zhang \cite{ZhouZhang2024}
established local well-posedness and a vanishing-diffusivity
limit in supercritical Besov spaces. Meng, Nie, Ye and Yin \cite{MengNieYeYin2024} obtained results
on global well-posedness near equilibrium, blow-up and norm inflation
in Sobolev spaces on the torus. Zheng, Zhou and Yang
\cite{ZhengZhouYang2025} investigated local well-posedness and the
inviscid limit for a general Keller--Segel system with logistic sources.
At the endpoint, Ding and Yu \cite{DingYu2026} established local
existence, uniqueness and continuous dependence in $B_{\infty,1}^{1}$,
and proved nonuniform dependence on the initial data. For related endpoint results for the Euler and Navier--Stokes
equations, we refer to Guo, Li and Yin \cite{GuoLiYin2019}.
Yu and Liu \cite{YuLiu2025} studied nonuniform convergence with respect
to the initial data in Besov spaces. The present paper addresses the
inviscid limit in $B_{\infty,1}^{1}$ and quantitative estimates in
lower-order norms.

For $1<p<\infty$, the embedding
$B_{p,1}^{1+d/p}\hookrightarrow W^{1,\infty}$ provides Lipschitz bounds,
and the error can be estimated in $B_{p,1}^{0}$ by a critical product
estimate. The case $p=1$ requires a separate argument. For $p=\infty$, we first estimate the error in
$L^\infty$ and then apply logarithmic interpolation to obtain a bound
in $B_{\infty,1}^{0}$. Strong convergence in $B_{\infty,1}^{1}$ follows
by regularizing the initial data and applying stability estimates
uniform in $\ep$.

We now state the main results.

\begin{theo}\label{thm:visc}
Let $d\ge1$, $1<p<\infty$, and
$u_0\in B_{p,1}^{1+d/p}(\R^d)$. Set
$M=4(1+\|u_0\|_{B_{p,1}^{1+d/p}})$.
Then there exists a time $T=T(d,p,M)>0$, independent of $\ep$,
such that, for every $0<\ep\le1$, the systems
\eqref{eq:PKS} and \eqref{eq:HKSE} admit unique solutions
\[
u^\ep,u^0\in C([0,T];B_{p,1}^{1+d/p}(\R^d))
\]
with the same initial datum $u_0$. Moreover,
\[
 \|u^\ep-u^0\|_{L_T^\infty(B_{p,1}^{0})}
 \le C_{d,p,M,T}
 \begin{cases}
 \ep,&p\le d,\\[1mm]
 \ep^{(1+d/p)/2},&p>d.
 \end{cases}
\]
The constant is independent of $\ep$.
The case $p=1$ is considered in
Proposition~\ref{prop:p1-boundary}.
\end{theo}

\begin{theo}\label{thm:c5-endpoint-vv}
Let $d\ge1$ and $u_0\in B_{\infty,1}^{1}(\R^d)$.
Let $T=T_*>0$ be the common existence time given by
Lemma~\ref{lem:uniform-local}, and let $u^\ep,u^0$ be the solutions of
\eqref{eq:PKS} and \eqref{eq:HKSE} with the same initial data $u_0$
on $[0,T]$.
Let $M=4(1+\|u_0\|_{B_{\infty,1}^{1}})$.
Then, for every $0<\ep\le1$,
\begin{equation}\label{eq:c5-rate-Linf}
\|u^\ep-u^0\|_{L_T^\infty(L^\infty)}
\le C_{M,T}\ep^{1/2},
\end{equation}
and
\begin{equation}\label{eq:c5-rate-B0}
\|u^\ep-u^0\|_{L_T^\infty(B_{\infty,1}^{0})}
\le C_{M,T}\ep^{1/2}
\left(1+\log\left(e+\frac{C_M}{\ep^{1/2}}\right)\right).
\end{equation}
Moreover,
\begin{equation}\label{eq:c5-strong-B1}
\lim_{\ep\to0}\|u^\ep-u^0\|_{L_T^\infty(B_{\infty,1}^{1})}=0.
\end{equation}

Furthermore, one can find a fixed initial datum
$u_0\in B_{\infty,1}^{1}(\R^d)$, independent of $\ep$,
such that, with $T$ and $M$ chosen as above for this datum,
the corresponding solutions satisfy the following two-sided estimate:
for every fixed $t_*\in(0,T]$, there exist
$c_{t_*}>0$ and $\ep_{t_*}\in(0,1]$, independent of $\ep$, such that
\begin{equation}\label{eq:c5-sharp-two-sided}
\frac{c_{t_*}\ep^{1/2}}{[\log(e/\ep)]^2}
\le \|u^\ep(t_*)-u^0(t_*)\|_{L^\infty}
\le C_{M,T}\ep^{1/2},
\qquad 0<\ep\le\ep_{t_*}.
\end{equation}
Consequently, the exponent $1/2$ in \eqref{eq:c5-rate-Linf}
is sharp: for this fixed datum, no convergence rate
$O(\ep^\alpha)$ with $\alpha>1/2$ holds in either
$L_T^\infty(L^\infty)$ or $L_T^\infty(B_{\infty,1}^{0})$.
\end{theo}

The common existence time and uniform a priori estimates follow from
Lemma~\ref{lem:uniform-local}. The convergence estimates remain valid
on any common existence interval on which the solutions have uniform
bounds in the corresponding critical Besov spaces.
Proposition~\ref{prop:higher-rates} gives improved rates for more regular
initial data. The sharpness assertion follows from
Proposition~\ref{prop:c5-lower-fixed}.

The organization of the paper is as follows. Section~\ref{sec:preliminaries}
recalls the properties of Besov spaces and the transport--diffusion
estimates used in the proofs. Section~\ref{sec:finite-p-vv} establishes
uniform local estimates and proves the convergence rates for $p<\infty$.
Section~\ref{sec:endpoint-vanishing-viscosity}
treats the case $p=\infty$, including convergence rates and lower bounds.

\section{Preliminaries}\label{sec:preliminaries}

We recall the Littlewood--Paley decomposition, Besov spaces and Bony's decomposition \cite{BCD2011}. All Besov spaces considered below are nonhomogeneous spaces on $\R^d$. Throughout, $C$ denotes a positive constant which may vary from line to line and is independent of $\ep$. We denote by $p'$ the conjugate exponent of $p$, with $1/p+1/p'=1$.

Choose radial functions $\chi,\varphi\in C_c^\infty(\R^d)$,
taking values in $[0,1]$, supported in
$\{|\xi|\le4/3\}$ and $\{3/4\le|\xi|\le8/3\}$, respectively,
such that
\[
 \chi(\xi)+\sum_{j\ge0}\varphi(2^{-j}\xi)=1,
 \qquad \xi\in\R^d.
\]
The dyadic supports have uniformly finite overlap.
For $f\in\mathcal S'(\R^d)$, define the dyadic blocks by
\[
 \Delta_jf=
 \begin{cases}
  0, & j\le-2,\\
  \chi(D)f, & j=-1,\\
  \varphi(2^{-j}D)f, & j\ge0,
 \end{cases}
\]
and the low-frequency cutoff operators by
\[
 S_Nf=\sum_{j<N}\Delta_jf,
 \qquad N\in\mathbb Z.
\]

For $s\in\R$ and $1\le p,r\le\infty$, the Besov space
$B_{p,r}^{s}(\R^d)$ is defined by
\[
 B_{p,r}^{s}(\R^d)
 =
 \left\{
 f\in\mathcal S'(\R^d):
 \|f\|_{B_{p,r}^{s}}<\infty
 \right\},
\]
where
\[
 \|f\|_{B_{p,r}^{s}}
 \triangleq
 \|
 \bigl(2^{js}\|\Delta_jf\|_{L^p}\bigr)_{j\ge-1}
 \|_{\ell^r}.
\]

For $T>0$ and $1\le q\le\infty$, we write
$L_T^q(X)=L^q(0,T;X)$.
The Chemin--Lerner norm is defined by
\[
 \|f\|_{\widetilde L_T^q(B_{p,r}^{s})}
 \triangleq
 \|
 \bigl(2^{js}\|\Delta_jf\|_{L_T^q(L^p)}\bigr)_{j\ge-1}
 \|_{\ell^r}.
\]
In particular, Tonelli's theorem gives
\[
 \|f\|_{\widetilde L_T^1(B_{p,1}^{s})}
 =
 \|f\|_{L_T^1(B_{p,1}^{s})}.
\]

We also recall Bony's decomposition
\[
 fg=T_fg+T_gf+R(f,g),
\]
where the paraproduct and remainder operators are given by
\[
 T_fg\triangleq
 \sum_{k\ge-1}S_{k-1}f\,\Delta_kg,
 \qquad
 R(f,g)\triangleq
 \sum_{\substack{k,\ell\ge-1\\|k-\ell|\le1}}
 \Delta_kf\,\Delta_\ell g.
\]
These expressions are well defined for the products
considered below.

Here and below, $D^m$ denotes the collection of derivatives
of order $m$.

\begin{lem}[\textit{Lemma~2.1, \cite{BCD2011}}]
\label{lem:bernstein}
Let $\mathcal C$ be an annulus and $\mathcal B$ a ball.
A constant $C$ exists such that for any nonnegative integer $k$,
any couple $(p,q)$ in $[1,\infty]^2$ with $q\ge p\ge1$,
and any function $u$ of $L^p$, we have
\[
 \operatorname{supp}\widehat u\subset\lambda\mathcal B
 \quad\Longrightarrow\quad
 \|D^ku\|_{L^q}
 \triangleq\sup_{|\alpha|=k}\|\partial^\alpha u\|_{L^q}
 \le C^{k+1}\lambda^{k+d(\frac1p-\frac1q)}
       \|u\|_{L^p},
\]
\[
 \operatorname{supp}\widehat u\subset\lambda\mathcal C
 \quad\Longrightarrow\quad
 C^{-k-1}\lambda^k\|u\|_{L^p}
 \le\|D^ku\|_{L^p}
 \le C^{k+1}\lambda^k\|u\|_{L^p}.
\]
\end{lem}

Bernstein's inequalities give the embeddings used below:
\begin{equation}\label{eq:prelim-embeddings}
 B_{p,1}^{d/p}\hookrightarrow L^\infty\quad(1\le p<\infty),
 \qquad B_{\infty,1}^{1}\hookrightarrow C_b^1\hookrightarrow\Lip.
\end{equation}
The function series and its first derivative series converge uniformly
for the second embedding. Differentiation also satisfies
\begin{equation}\label{eq:prelim-derivative}
 \|\nabla f\|_{B_{p,1}^{s-1}}\le C\|f\|_{B_{p,1}^{s}},
 \qquad s\in\R,\quad1\le p\le\infty.
\end{equation}
Indeed, apply Bernstein's inequalities to each dyadic block, including the low block,
multiply by $2^{j(s-1)}$, and sum. We also recall the interpolation
theorem used for the convergence rates.

\begin{theo}[\textit{Theorem~2.80, \cite{BCD2011}}]
\label{thm:bcd-interpolation}
If $s_1$ and $s_2$ are real numbers such that $s_1<s_2$,
$\theta\in(0,1)$, and $(p,r)$ is in $[1,\infty]^2$, then we have
\begin{equation}\label{eq:besov-interpolation}
 \|u\|_{B_{p,r}^{\theta s_1+(1-\theta)s_2}}
 \le \|u\|_{B_{p,r}^{s_1}}^\theta
      \|u\|_{B_{p,r}^{s_2}}^{1-\theta}.
\end{equation}
\end{theo}

\begin{lem}\label{lem:log}
Let $s>0$, $1\le p\le\infty$, and
$f\in B_{p,\infty}^{0}\cap B_{p,\infty}^{s}$. Then
\begin{equation*}
 \|f\|_{B_{p,1}^{0}}
 \le C_s \|f\|_{B_{p,\infty}^{0}}\left(1+\log\left(e+\frac{\|f\|_{B_{p,\infty}^{s}}}{\|f\|_{B_{p,\infty}^{0}}}\right)\right).
\end{equation*}
\end{lem}

\begin{proof}
%We may assume that $\|f\|_{B_{p,\infty}^{0}}>0$.
Choosing a positive integer $N$, which will be determined later, we have
\[
\begin{aligned}
\|f\|_{B_{p,1}^{0}}
&=\sum_{-1\le j\le N-1}\|\Delta_jf\|_{L^p}
  +\sum_{j\ge N}\|\Delta_jf\|_{L^p}\\
&\le (N+1)\|f\|_{B_{p,\infty}^{0}}
  +\frac{2^{-(N-1)s}}{2^s-1}\|f\|_{B_{p,\infty}^{s}},
\end{aligned}
\]
Taking
\[
N=1+\left\lfloor
\frac{1}{s}\log_2\left(
e+\frac{\|f\|_{B_{p,\infty}^{s}}}
       {\|f\|_{B_{p,\infty}^{0}}}
\right)\right\rfloor
\]
yields the desired inequality.
\end{proof}

\begin{cor}[\textit{Corollary~2.86, \cite{BCD2011}}]
\label{lem:bony}
For any positive real number $s$ and any $(p,r)$ in $[1,\infty]^2$,
the space $L^\infty\cap B_{p,r}^{s}$ is an algebra, and a constant
$C$ exists such that
\begin{equation}\label{eq:c5-tame-tools}
 \|uv\|_{B_{p,r}^{s}}
 \le \frac{C^{s+1}}{s}
 \left(
 \|u\|_{L^\infty}\|v\|_{B_{p,r}^{s}}
 +\|u\|_{B_{p,r}^{s}}\|v\|_{L^\infty}
 \right).
\end{equation}
\end{cor}

\begin{lem}\label{lem:elliptic}
For $s\in\R$ and $1\le p\le\infty$,
\begin{equation}\label{eq:c5-multiplier-tools}
 \|(I-\Delta)^{-1}f\|_{B_{p,1}^{s+2}}+
 \|\nabla (I-\Delta)^{-1}f\|_{B_{p,1}^{s+1}}+
 \|\Delta (I-\Delta)^{-1}f\|_{B_{p,1}^{s}}
 \le C_s\|f\|_{B_{p,1}^{s}}.
\end{equation}
At the $L^p$ level,
\begin{equation}\label{eq:c5-Linf-tools}
 \|(I-\Delta)^{-1}f\|_{L^p}\le\|f\|_{L^p},\quad
 \|\nabla (I-\Delta)^{-1}f\|_{L^p}\le C\|f\|_{L^p},\quad
 \|\Delta (I-\Delta)^{-1}f\|_{L^p}\le2\|f\|_{L^p}.
\end{equation}
\end{lem}

\begin{proof}
The first estimates follow from \cite{BCD2011}
applied to the smooth symbols
$(1+|\xi|^2)^{-1}$, $i\xi(1+|\xi|^2)^{-1}$, and
$-|\xi|^2(1+|\xi|^2)^{-1}$ on each dyadic block, including the low block.
For \eqref{eq:c5-Linf-tools}, use
\[
 (I-\Delta)^{-1}f=\int_0^\infty e^{-t}e^{t\Delta}f\,dt,
 \qquad \|\nabla e^{t\Delta}f\|_{L^p}
 \le Ct^{-1/2}\|f\|_{L^p}.
\]
The Laplacian estimate follows from the exact identity $\Delta (I-\Delta)^{-1}=(I-\Delta)^{-1}-I$.
\end{proof}

Consider
\begin{equation}\label{transport}
 \partial_tf+a\cdot\nabla f-\ep\Delta f=h,\qquad f(0)=f_0,
 \qquad \ep\ge0,
\end{equation}
where $a\in C^{0,1}$, with the displayed norms
integrable in time. All estimates below may first be proved for smooth
solutions, then passed to solutions of the corresponding linear equation
by regularization and lower semicontinuity.

\begin{lem}[\textit{Proof of Proposition~2.1, \cite{Danchin2005}}]
\label{lem:scalar-estimate}
Let $1\le p\le\infty$. Solutions of \eqref{transport} satisfy
\[
 \|f(t)\|_{L^p}
 \le \|f_0\|_{L^p}
 +\int_0^t\|h(\tau)\|_{L^p}\,d\tau
 +\frac1p\int_0^t\|\diver a(\tau)\|_{L^\infty}
                     \|f(\tau)\|_{L^p}\,d\tau.
\]
Here $1/p=0$ when $p=\infty$.
The estimate is uniform for $\ep\ge0$.
\end{lem}

\begin{lem}[\textit{Lemma~2.100,\cite{BCD2011}}]
\label{lem:bcd-commutator}\label{lem:endpoint-transport}
Let $v$ be a smooth vector field and $f$ a smooth scalar function
on $\R^d$. For $j\ge-1$, set
\[
R_j=[v\cdot\nabla,\Delta_j]f.
\]
For $1\le p<\infty$ and $0\le\sigma<1+d/p$, excluding
$(p,\sigma)=(1,0)$, we have
\[
\sum_{j\ge-1}2^{j\sigma}\|R_j\|_{L^p}
\le C\|\nabla v\|_{B_{p,1}^{d/p}}
      \|f\|_{B_{p,1}^{\sigma}}.
\]
For $1\le p\le\infty$ and $\sigma>0$, we also have
\begin{equation}\label{eq:bcd-commutator-tame}
\begin{aligned}
\sum_{j\ge-1}2^{j\sigma}\|R_j\|_{L^p}
\le C\left(
\|\nabla v\|_{L^\infty}\|f\|_{B_{p,1}^{\sigma}}
+\|\nabla f\|_{L^\infty}
 \|\nabla v\|_{B_{p,1}^{\sigma-1}}
\right).
\end{aligned}
\end{equation}
The constants depend only on $d,p,\sigma$.
\end{lem}

For bounded Lipschitz $a,f$ with the norms below finite,
taking $p=\infty$ and $\sigma=s>0$ in
\eqref{eq:bcd-commutator-tame}, and using
\eqref{eq:prelim-derivative}, gives
\begin{equation}\label{eq:c5-commutator-tools}
\begin{aligned}
\sum_{j\ge-1}2^{js}\|[a\cdot\nabla,\Delta_j]f\|_{L^\infty}
\le C_s\left(
\|\nabla a\|_{L^\infty}\|f\|_{B_{\infty,1}^{s}}
+\|\nabla f\|_{L^\infty}\|a\|_{B_{\infty,1}^{s}}
\right).
\end{aligned}
\end{equation}
Applying the dyadic blocks to \eqref{transport}, using
Lemma~\ref{lem:scalar-estimate} with $p=\infty$, and summing
with weights $2^j$, we obtain
\begin{equation}\label{eq:c5-critical-transport-tools}
\begin{aligned}
\|f(t)\|_{B_{\infty,1}^{1}}
\le \|f_0\|_{B_{\infty,1}^{1}}
+\int_0^t\|h(\tau)\|_{B_{\infty,1}^{1}}\,d\tau+C\int_0^t\|a(\tau)\|_{B_{\infty,1}^{1}}\|f(\tau)\|_{B_{\infty,1}^{1}}\,d\tau,
\end{aligned}
\end{equation}
uniformly for $\ep\ge0$.
Here we have used the embedding
$B_{\infty,1}^{1}\hookrightarrow W^{1,\infty}$
to control both Lipschitz seminorms.

\begin{lem}[\textit{Theorem~3.14, \cite{BCD2011}}]
\label{lem:transport-estimate}
Let $1\le p\le p_1\le\infty$ and $1\le r\le\infty$. Assume that
\begin{equation}\label{eq:bcd-transport-index}
\sigma\ge-d\min\left\{\frac1{p_1},\frac1{p'}\right\}
\quad\text{or}\quad
\sigma\ge-1-d\min\left\{\frac1{p_1},\frac1{p'}\right\}
\quad\text{if }\diver a=0,
\end{equation}
with strict inequality if $r<\infty$.

There exists a constant $C=C(d,p,p_1,r,\sigma)$,
such that for all solutions
$f\in L^\infty([0,T];B_{p,r}^{\sigma})$
of \eqref{transport} with $\ep=0$, initial data
$f_0\in B_{p,r}^{\sigma}$, and
$h\in L^1([0,T];B_{p,r}^{\sigma})$, we have, for almost every
$t\in[0,T]$,
\begin{equation}\label{eq:bcd-transport-estimate}
\begin{aligned}
\|f\|_{\widetilde L_t^\infty(B_{p,r}^{\sigma})}
&\le
\left(
\|f_0\|_{B_{p,r}^{\sigma}}
+\int_0^t e^{-CV_{p_1}(\tau)}
\|h(\tau)\|_{B_{p,r}^{\sigma}}\,d\tau
\right)e^{CV_{p_1}(t)}
\end{aligned}
\end{equation}
with, if the inequality is strict in
\eqref{eq:bcd-transport-index},
\[
V_{p_1}'(t)=
\begin{cases}
\displaystyle
\|\nabla a(t)\|_{B_{p_1,\infty}^{d/p_1}\cap L^\infty},
&\displaystyle \sigma<1+\frac d{p_1},\\[3mm]
\displaystyle
\|\nabla a(t)\|_{B_{p_1,r}^{\sigma-1}},
&
\displaystyle \sigma>1+\frac d{p_1},
\quad\text{or}\quad \{
\sigma=1+\frac d{p_1}\ \text{ and }r=1\}
\end{cases}
\]
and, if equality holds in \eqref{eq:bcd-transport-index} and $r=\infty$,
then
\[
V_{p_1}'(t)
=\|\nabla a(t)\|_{B_{p_1,1}^{d/p_1}}.
\]
If $f=a$, then for all $\sigma>0$
($\sigma>-1$ if $\diver a=0$),
the estimate \eqref{eq:bcd-transport-estimate} holds with
\[
V_{p_1}'(t)=\|\nabla a(t)\|_{L^\infty}.
\]
\end{lem}

\begin{lem}[\textit{Theorem~3.38, \cite{BCD2011}}]
\label{thm:bcd-diffusion}\label{lem:transport-diffusion}
Let $1\le p_1\le p\le\infty$, $1\le r\le\infty$, and let
$s\in\mathbb R$ satisfy \eqref{eq:bcd-transport-index}
with $\sigma$ replaced by $s$, with strict inequality if $r<\infty$.
Let $V_{p_1}$ be defined as in Lemma~\ref{lem:transport-estimate},
with $\sigma=s$.

There exists a constant $C$, depending only on
$d,r,s$ and $s-1-d/p_1$, such that for any smooth solution
$f$ of \eqref{transport} and any
$1\le\rho_1\le\rho\le\infty$, we have
\[
\begin{aligned}
\ep^{1/\rho}
\|f\|_{\widetilde L_T^\rho(B_{p,r}^{s+2/\rho})}
&\le
Ce^{C(1+\ep T)^{1/\rho}V_{p_1}(T)}
\Biggl(
(1+\ep T)^{1/\rho}\|f_0\|_{B_{p,r}^{s}}\\
&\qquad\qquad
+(1+\ep T)^{1+1/\rho-1/\rho_1}
\ep^{1/\rho_1-1}
\|h\|_{\widetilde L_T^{\rho_1}
(B_{p,r}^{s-2+2/\rho_1})}
\Biggr).
\end{aligned}
\]
\end{lem}

\begin{lem}[\textit{Comparison principle, \cite[Chapter~III, Theorem~10]{ProtterWeinberger1984}}]
\label{lem:comparison}
Let $T>0$ and $\ep>0$. Assume that $a\in C^{0,1}$ in space
and that $a$ and $b$ are bounded on $[0,T]\times\R^d$.
Let $r,Q$ be continuous on $[0,T]\times\R^d$ and classical
for $t>0$, satisfying
\[
\begin{aligned}
 \partial_t r-\ep\Delta r+a\cdot\nabla r-br&=h,\\
 \partial_t Q-\ep\Delta Q+a\cdot\nabla Q-bQ&\le h.
\end{aligned}
\]
Assume that $r$ is bounded and that $Q(t,x)\to-\infty$
as $|x|\to\infty$, uniformly for $t\in[0,T]$.
If $Q(0,x)\le r(0,x)$ for all $x\in\R^d$, then
\[
 Q(t,x)\le r(t,x)
 \qquad\text{on }[0,T]\times\R^d.
\]
If the differential and initial inequalities for $Q$ are reversed
and $Q(t,x)\to+\infty$ uniformly in time as $|x|\to\infty$,
then $Q\ge r$.
\end{lem}

\begin{lem}[\textit{Dini criterion, \cite[Lemma~30(i)]{KochSickel2002}}]
\label{lem:dini}
Let $f\in C_b^1(\R^d)$ and define
\[
 \omega(\nabla f,h)
 =\sup_{|x-y|\le h}|\nabla f(x)-\nabla f(y)|.
\]
If
\[
 \int_0^1\frac{\omega(\nabla f,h)}h\,dh<\infty,
\]
then $f\in B_{\infty,1}^{1}$.
The assertion includes bounded periodic functions.
\end{lem}

\begin{proof}
By the cited result, each component of $\nabla f$ belongs to
$B_{\infty,1}^{0}$. Bernstein's inequalities yield
\[
 \|f\|_{B_{\infty,1}^{1}}
 \le C\left(
 \|f\|_{L^\infty}
 +\sum_{k=1}^{d}\|\partial_k f\|_{B_{\infty,1}^{0}}
 \right)<\infty.
\]
\end{proof}

\section[Inviscid limit for finite p]{Inviscid limit for $1\le p<\infty$}
\label{sec:finite-p-vv}

We first establish a common existence time and uniform estimates for
$1\le p\le\infty$. These estimates will also be used in
Section~\ref{sec:endpoint-vanishing-viscosity}. We then prove the
convergence rates for finite $p$.

\begin{lem}\label{lem:uniform-local}
Let $d\ge1$, $1\le p\le\infty$, and
$u_0\in B_{p,1}^{1+d/p}(\R^d)$. Set
\[
 R=1+\|u_0\|_{B_{p,1}^{1+d/p}}.
\]
There exists a constant $C=C(d,p)\ge1$ such that, with
\[
 T_*=\frac{1}{16C(R+R^2)},
\]
the solution $u^\ep$ of \eqref{eq:c5-unified} exists uniquely in
$C([0,T_*];B_{p,1}^{1+d/p})$ and, for all $\ep\in[0,1]$
and $t\in[0,T_*]$,
\begin{equation}\label{eq:original-uniform-local}
 \|u^\ep(t)\|_{B_{p,1}^{1+d/p}}\le2R.
\end{equation}
For each fixed $\ep\in[0,1]$, boundedness of the critical norm
implies continuation.
\end{lem}

\begin{proof}
For $\ep=0$, local well-posedness is established in
\cite{ZhangMuZhou2022} when $p<\infty$ and in
\cite{DingYu2026} when $p=\infty$.
For each fixed $\ep>0$, local well-posedness in
$B_{p,1}^{1+d/p}$ follows by adapting the local constructions
in \cite{ZhangMuZhou2022,DingYu2026}, using the
transport--diffusion estimates in place of the transport estimates.
For each fixed $\ep\in[0,1]$, the local existence time depends
only on $\ep$ and an upper bound for the initial critical norm.
Thus boundedness of the critical norm implies continuation.

By Corollary~\ref{lem:bony} and Lemma~\ref{lem:elliptic},
\[
\begin{aligned}
 \|(1-2u^\ep)\nabla(I-\Delta)^{-1}u^\ep\|_{B_{p,1}^{1+d/p}}
 &\le C\left(
 \|u^\ep\|_{B_{p,1}^{1+d/p}}
 +\|u^\ep\|_{B_{p,1}^{1+d/p}}^2\right),\\
 \|u^\ep(1-u^\ep)
 \bigl((I-\Delta)^{-1}u^\ep-u^\ep\bigr)\|_{B_{p,1}^{1+d/p}}
 &\le C\left(
 \|u^\ep\|_{B_{p,1}^{1+d/p}}^2
 +\|u^\ep\|_{B_{p,1}^{1+d/p}}^3\right).
\end{aligned}
\]
Applying $\Delta_j$ to \eqref{eq:c5-unified}, using
Lemma~\ref{lem:scalar-estimate} and
\eqref{eq:bcd-commutator-tame}, and summing with weights
$2^{j(1+d/p)}$, we obtain, by the embedding
$B_{p,1}^{1+d/p}\hookrightarrow W^{1,\infty}$,
\[
\begin{aligned}
 \|u^\ep(t)\|_{B_{p,1}^{1+d/p}}
 \le \|u^\ep(0)\|_{B_{p,1}^{1+d/p}}+C\int_0^t
 \|u^\ep(\tau)\|_{B_{p,1}^{1+d/p}}^2
 \left(1+\|u^\ep(\tau)\|_{B_{p,1}^{1+d/p}}\right)\,d\tau.
\end{aligned}
\]
The diffusion term is dissipative, and $C=C(d,p)\ge1$
is independent of $\ep$.
The estimate holds for critical solutions by regularization
and lower semicontinuity.

Since
\[
 \|u^\ep(0)\|_{B_{p,1}^{1+d/p}}
 =\|u_0\|_{B_{p,1}^{1+d/p}}=R-1,
\]
whenever the bound
\[
\|u^\ep(t)\|_{B_{p,1}^{1+d/p}}\le2R
\]
holds on an initial interval of existence contained in $[0,T_*]$,
we have
\[
\begin{aligned}
 \|u^\ep(t)\|_{B_{p,1}^{1+d/p}}
 &\le R-1+Ct(4R^2+8R^3)\\
 &\le R-1+\frac{4R^2+8R^3}{16(R+R^2)}\\
 &\le2R.
\end{aligned}
\]
Continuity closes the bootstrap. If a maximal existence time
were at or before $T_*$, the continuation criterion would extend
the solution, a contradiction. Hence all solutions exist on
$[0,T_*]$ and satisfy \eqref{eq:original-uniform-local}.
\end{proof}

We shall use the following product estimates to control the
nonlinear terms in \eqref{eq:error-equation}.

\begin{lem}\label{lem:critical-multiplier}
If $1<p<\infty$, $f\in B_{p,1}^{0}$, and $g\in B_{p,1}^{d/p}$, then
\begin{equation}\label{eq:critical-multiplier}
 \|fg\|_{B_{p,1}^{0}}
 \le C_{d,p}\|f\|_{B_{p,1}^{0}}\|g\|_{B_{p,1}^{d/p}}.
\end{equation}
If $p=1$ and $0<\beta<d$, then instead
\begin{equation}\label{eq:p1-positive-multiplier}
 \|fg\|_{B_{1,1}^{\beta}}
 \le C_{d,\beta}\|f\|_{B_{1,1}^{\beta}}\|g\|_{B_{1,1}^{d}}.
\end{equation}
\end{lem}

\begin{proof}
We use Bony's decomposition
\[
fg=T_gf+T_fg+R(f,g)
\]
and the continuity estimates in
\cite[Theorems~2.82 and~2.85]{BCD2011}.

First, let $1<p<\infty$. The embeddings
\[
B_{p,1}^{d/p}\hookrightarrow L^\infty,
\qquad
B_{p,1}^{0}\hookrightarrow B_{\infty,\infty}^{-d/p}
\]
and the paraproduct estimates give
\[
\begin{aligned}
\|T_gf\|_{B_{p,1}^{0}}
&\le C\|g\|_{L^\infty}\|f\|_{B_{p,1}^{0}}
\le C\|f\|_{B_{p,1}^{0}}\|g\|_{B_{p,1}^{d/p}},\\
\|T_fg\|_{B_{p,1}^{0}}
&\le C\|f\|_{B_{\infty,\infty}^{-d/p}}
       \|g\|_{B_{p,1}^{d/p}}
\le C\|f\|_{B_{p,1}^{0}}\|g\|_{B_{p,1}^{d/p}}.
\end{aligned}
\]

For the remainder, we distinguish two cases.
If $p\ge2$, then $d/p>0$, and the remainder estimate,
together with
$B_{p/2,1}^{d/p}\hookrightarrow B_{p,1}^{0}$, yields
\[
\begin{aligned}
\|R(f,g)\|_{B_{p,1}^{0}}
&\le C\|R(f,g)\|_{B_{p/2,1}^{d/p}}\\
&\le C\|f\|_{B_{p,1}^{0}}
       \|g\|_{B_{p,\infty}^{d/p}}\\
&\le C\|f\|_{B_{p,1}^{0}}
       \|g\|_{B_{p,1}^{d/p}}.
\end{aligned}
\]
If $1<p<2$, then $p<p'$ and
\[
B_{p,1}^{d/p}\hookrightarrow B_{p',\infty}^{d/p'},
\qquad
B_{1,1}^{d/p'}\hookrightarrow B_{p,1}^{0}.
\]
Since $1/p+1/p'=1$ and $d/p'>0$, we obtain
\[
\begin{aligned}
\|R(f,g)\|_{B_{p,1}^{0}}
&\le C\|R(f,g)\|_{B_{1,1}^{d/p'}}\\
&\le C\|f\|_{B_{p,1}^{0}}
       \|g\|_{B_{p',\infty}^{d/p'}}\\
&\le C\|f\|_{B_{p,1}^{0}}
       \|g\|_{B_{p,1}^{d/p}}.
\end{aligned}
\]
Combining these estimates proves \eqref{eq:critical-multiplier}.

Now let $p=1$ and $0<\beta<d$. We use the embeddings
\[
B_{1,1}^{d}\hookrightarrow L^\infty\cap B_{\infty,\infty}^{0},
\qquad
B_{1,1}^{\beta}\hookrightarrow B_{\infty,\infty}^{\beta-d}.
\]
Since $\beta-d<0$, the paraproduct estimates imply
\[
\begin{aligned}
\|T_gf\|_{B_{1,1}^{\beta}}
&\le C\|g\|_{L^\infty}\|f\|_{B_{1,1}^{\beta}}
\le C\|f\|_{B_{1,1}^{\beta}}\|g\|_{B_{1,1}^{d}},\\
\|T_fg\|_{B_{1,1}^{\beta}}
&\le C\|f\|_{B_{\infty,\infty}^{\beta-d}}
       \|g\|_{B_{1,1}^{d}}
\le C\|f\|_{B_{1,1}^{\beta}}\|g\|_{B_{1,1}^{d}}.
\end{aligned}
\]
Finally, $\beta>0$ allows us to apply the remainder estimate:
\[
\begin{aligned}
\|R(f,g)\|_{B_{1,1}^{\beta}}
&\le C\|f\|_{B_{1,1}^{\beta}}
       \|g\|_{B_{\infty,\infty}^{0}}\\
&\le C\|f\|_{B_{1,1}^{\beta}}
       \|g\|_{B_{1,1}^{d}}.
\end{aligned}
\]
Summing the three bounds proves
\eqref{eq:p1-positive-multiplier}.
\end{proof}

\begin{proof}[Proof of Theorem~\ref{thm:visc}]
Let $u^\ep$ and $u^0$ be the solutions of \eqref{eq:PKS} and
\eqref{eq:HKSE} with the same initial data $u_0$. Their chemical signals satisfy
\[
S^\ep=(I-\Delta)^{-1}u^\ep,\qquad S^0=(I-\Delta)^{-1}u^0.
\]
The difference of the transport terms can be written as
\[
\begin{split}
&(1-2u^\ep)\nabla u^\ep\cdot\nabla S^\ep
 -(1-2u^0)\nabla u^0\cdot\nabla S^0\\
=&(1-2u^0)\nabla S^0\cdot\nabla \bigl(u^\ep-u^0\bigr)
 +(1-2u^\ep)\nabla u^\ep\cdot\nabla \bigl(S^\ep-S^0\bigr)
 -2\bigl(u^\ep-u^0\bigr)\nabla S^0\cdot\nabla u^\ep.
\end{split}
\]
For the remaining nonlinear terms, we have
\[
u^\ep(1-u^\ep)\Delta S^\ep-u^0(1-u^0)\Delta S^0
=\bigl(u^\ep-u^0\bigr)(1-u^\ep-u^0)\Delta S^\ep+u^0(1-u^0)\Delta \bigl(S^\ep-S^0\bigr).
\]
Subtracting \eqref{eq:HKSE} from \eqref{eq:PKS}
and using the preceding identities, we obtain
\begin{equation}\label{eq:error-equation}
\begin{cases}
\displaystyle
\partial_t(u^\ep-u^0)
+\bigl((1-2u^0)\nabla S^0\bigr)\cdot\nabla(u^\ep-u^0)
=\sum_{i=1}^5 F_i,\\[2mm]
(u^\ep-u^0)|_{t=0}=0,
\end{cases}
\end{equation}
where
\begin{equation}\label{eq:error-sources}
\begin{alignedat}{2}
F_1&=\ep\Delta u^\ep,\qquad
&F_2&=-(1-2u^\ep)\nabla u^\ep\cdot\nabla\bigl(S^\ep-S^0\bigr),\\[1mm]
F_3&=2\bigl(u^\ep-u^0\bigr)\nabla S^0\cdot\nabla u^\ep,\qquad
&F_4&=-\bigl(u^\ep-u^0\bigr)(1-u^\ep-u^0)\Delta S^\ep,\\[1mm]
F_5&=-u^0(1-u^0)\Delta\bigl(S^\ep-S^0\bigr).
\end{alignedat}
\end{equation}
These identities will also be used in
Section~\ref{sec:endpoint-vanishing-viscosity}.

By \eqref{eq:original-uniform-local} and
Lemma~\ref{lem:bernstein}, Corollary~\ref{lem:bony} and Lemma~\ref{lem:elliptic},
\[
\|S^\ep\|_{L_T^\infty(B_{p,1}^{3+d/p})}
+\|S^0\|_{L_T^\infty(B_{p,1}^{3+d/p})}\le C_M,
\qquad
\|\nabla \bigl((1-2u^0)\nabla S^0\bigr)\|_{L_T^\infty(B_{p,1}^{d/p})}\le C_M.
\]
Indeed, we write
\[
\nabla\bigl((1-2u^0)\nabla S^0\bigr)
=-2\nabla u^0\otimes\nabla S^0
 +(1-2u^0)\nabla^2S^0
\]
and estimate the products on the right-hand side in $B_{p,1}^{d/p}$.
Applying Lemma~\ref{lem:transport-estimate} to
\eqref{eq:error-equation} with $\sigma=0$ yields
\begin{equation}\label{eq:finite-error-bound}
\|u^\ep(t)-u^0(t)\|_{B_{p,1}^{0}}
\le C_{M,T}\int_0^t\sum_{i=1}^5
       \|F_i(\tau)\|_{B_{p,1}^{0}}\,d\tau.
\end{equation}
To estimate $F_1$, we use \eqref{eq:c5-unified} with $u=u^\ep$
\[
\partial_tu^\ep+\bigl((1-2u^\ep)\nabla S^\ep\bigr)\cdot\nabla u^\ep-\ep\Delta u^\ep=-u^\ep(1-u^\ep)\Delta S^\ep.
\]
The same product and elliptic estimates give
\[
\|\nabla \bigl((1-2u^\ep)\nabla S^\ep\bigr)\|_{L_T^\infty(B_{p,1}^{d/p})}
+\|-u^\ep(1-u^\ep)\Delta S^\ep\|_{L_T^\infty(B_{p,1}^{1+d/p})}\le C_M.
\]
Lemma~\ref{lem:transport-diffusion} therefore implies
\[
\ep\|u^\ep\|_{L_T^1(B_{p,1}^{3+d/p})}\le C_{M,T}.
\]

We first estimate $F_1$. If $p\le d$, then $1+d/p\ge2$. By \eqref{eq:original-uniform-local},
\[
\int_0^T\|F_1(t)\|_{B_{p,1}^{0}}\,dt
\le C\ep\|u^\ep\|_{L_T^1(B_{p,1}^{2})}
\le C_{M,T}\ep.
\]
If $p>d$, let $\theta=(1+d/p)/2$. By \eqref{eq:besov-interpolation} and H\"older's inequality, we obtain
\[
\begin{aligned}
\|F_1\|_{L_T^1(B_{p,1}^{0})}
&\le C\ep\|u^\ep\|_{L_T^1(B_{p,1}^{2})}\\
&\le C\ep T^\theta
\|u^\ep\|_{L_T^\infty(B_{p,1}^{1+d/p})}^{\theta}
\|u^\ep\|_{L_T^1(B_{p,1}^{3+d/p})}^{1-\theta}\\
&\le C_{M,T}\ep^\theta
=C_{M,T}\ep^{(1+d/p)/2}.
\end{aligned}
\]
To estimate the nonlinear terms $F_i$, $2\le i\le5$,
we first apply Lemma~\ref{lem:elliptic} to obtain
\[
\|\nabla(S^\ep-S^0)\|_{B_{p,1}^{0}}
+\|\Delta(S^\ep-S^0)\|_{B_{p,1}^{0}}
\le C\|u^\ep-u^0\|_{B_{p,1}^{0}}.
\]
By Corollary~\ref{lem:bony}, Lemma~\ref{lem:critical-multiplier},
\eqref{eq:c5-multiplier-tools}, and \eqref{eq:original-uniform-local},
we obtain
\begin{align*}
\|F_2\|_{B_{p,1}^{0}}
&\le C\|(1-2u^\ep)\nabla u^\ep\|_{B_{p,1}^{d/p}}
       \|\nabla(S^\ep-S^0)\|_{B_{p,1}^{0}}\\
&\le C\left(1+\|u^\ep\|_{B_{p,1}^{d/p}}\right)
       \|\nabla u^\ep\|_{B_{p,1}^{d/p}}
       \|u^\ep-u^0\|_{B_{p,1}^{0}}\\
&\le C\left(1+\|u^\ep\|_{B_{p,1}^{1+d/p}}\right)
       \|u^\ep\|_{B_{p,1}^{1+d/p}}
       \|u^\ep-u^0\|_{B_{p,1}^{0}}\\
&\le C_M\|u^\ep-u^0\|_{B_{p,1}^{0}}.
\end{align*}
The terms $F_3$, $F_4$, and $F_5$ are estimated similarly. 
Consequently,
\[
\sum_{i=2}^{5}\|F_i(t)\|_{B_{p,1}^{0}}
\le C_M\|u^\ep(t)-u^0(t)\|_{B_{p,1}^{0}}.
\]
Inserting these estimates into \eqref{eq:finite-error-bound} and
applying Gronwall's inequality completes the proof.
\end{proof}

\begin{prop}\label{prop:p1-boundary}
Let $u_0\in B_{1,1}^{1+d}$, let $T=T_*$ be given by
Lemma~\ref{lem:uniform-local}, and set
$M=4(1+\|u_0\|_{B_{1,1}^{1+d}})$.
Let $u^\ep$ and $u^0$ solve \eqref{eq:PKS} and \eqref{eq:HKSE} with
the same initial data $u_0$. Then, for $0<\ep\le1$,
\begin{equation}\label{eq:p1-L1-rate}
\|u^\ep-u^0\|_{L_T^\infty(L^1)}\le C_{M,T}\ep,
\end{equation}
and
\[
\|u^\ep-u^0\|_{L_T^\infty(B_{1,1}^{0})}
\le C_{M,T}
\begin{cases}
\ep\bigl(1+\log(e/\ep)\bigr),&d=1,\\
\ep,&d\ge2.
\end{cases}
\]
The constants are independent of $\ep$.
\end{prop}

\begin{proof}
By the embedding $B_{1,1}^{1+d}\hookrightarrow W^{1,\infty}$,
Lemma~\ref{lem:elliptic}, and
\eqref{eq:original-uniform-local}, we have
\[
\begin{aligned}
\|u^\ep\|_{L^\infty}+\|u^0\|_{L^\infty}
+\|\nabla u^\ep\|_{L^\infty}
+\|\nabla u^0\|_{L^\infty}
&\le C\left(
\|u^\ep\|_{B_{1,1}^{1+d}}
+\|u^0\|_{B_{1,1}^{1+d}}
\right)\le C_M,\\[1ex]
\|\nabla S^0\|_{L^\infty}
+\|\Delta S^\ep\|_{L^\infty}
+\|\Delta S^0\|_{L^\infty}
&\le C\left(
\|u^\ep\|_{L^\infty}+\|u^0\|_{L^\infty}
\right)\le C_M,\\[1ex]
\|\nabla(S^\ep-S^0)\|_{L^1}
+\|\Delta(S^\ep-S^0)\|_{L^1}
&\le C\|u^\ep-u^0\|_{L^1}.
\end{aligned}
\]
Arguing as in the proof of Theorem~\ref{thm:visc},
but using H\"older's inequality in place of
\eqref{eq:critical-multiplier}, we obtain
\[
\sum_{i=2}^{5}\|F_i\|_{L^1}
\le C_M\|u^\ep-u^0\|_{L^1}.
\]
Moreover,
\[
\begin{aligned}
\|\diver\bigl((1-2u^0)\nabla S^0\bigr)\|_{L^\infty}
&=\|-2\nabla u^0\cdot\nabla S^0
       +(1-2u^0)\Delta S^0\|_{L^\infty}\\
&\le 2\|\nabla u^0\|_{L^\infty}
       \|\nabla S^0\|_{L^\infty}+\left(1+2\|u^0\|_{L^\infty}\right)
       \|\Delta S^0\|_{L^\infty}\\
&\le C_M.
\end{aligned}
\]
We also have
\[
\begin{aligned}
\|F_1\|_{L_T^1(L^1)}
&=\ep\|\Delta u^\ep\|_{L_T^1(L^1)}\\
&\le C\ep\|u^\ep\|_{L_T^1(B_{1,1}^{2})}\\
&\le C\ep T\|u^\ep\|_{L_T^\infty(B_{1,1}^{1+d})}\\
&\le C_{M,T}\ep.
\end{aligned}
\]
Applying Lemma~\ref{lem:scalar-estimate} with $p=1$
to \eqref{eq:error-equation}, we obtain
\[
\begin{aligned}
\|u^\ep-u^0\|_{L_t^\infty(L^1)}
&\le \sum_{i=1}^{5}\|F_i\|_{L_t^1(L^1)}
+\int_0^t
\|\diver\bigl((1-2u^0)\nabla S^0\bigr)\|_{L^\infty}
\|u^\ep-u^0\|_{L_\tau^\infty(L^1)}\,d\tau\\
&\le C_{M,T}\ep
+C_M\int_0^t
\|u^\ep-u^0\|_{L_\tau^\infty(L^1)}\,d\tau.
\end{aligned}
\]
Gronwall's inequality yields
\[
\|u^\ep-u^0\|_{L_T^\infty(L^1)}
\le C_{M,T}\ep e^{C_M T}
\le C_{M,T}\ep.
\]

If $d=1$, applying Lemma~\ref{lem:log} with $p=1$ and $s=2$
and using the embeddings $L^1\hookrightarrow B_{1,\infty}^{0}$
and $B_{1,1}^{2}\hookrightarrow B_{1,\infty}^{2}$, we obtain
\[
\begin{aligned}
\|u^\ep-u^0\|_{B_{1,1}^{0}}
&\le C\|u^\ep-u^0\|_{B_{1,\infty}^{0}}
\left(
1+\log\left(
e+\frac{\|u^\ep-u^0\|_{B_{1,\infty}^{2}}}
{\|u^\ep-u^0\|_{B_{1,\infty}^{0}}}
\right)\right)\\
&\le C\|u^\ep-u^0\|_{L^1}
\left(
1+\log\left(
e+\frac{C_M}{\|u^\ep-u^0\|_{L^1}}
\right)\right)\\
&\le C_{M,T}\ep\bigl(1+\log(e/\ep)\bigr).
\end{aligned}
\]
Here we have used \eqref{eq:original-uniform-local},
\eqref{eq:p1-L1-rate}, and the fact that the function $\Phi(x)=x(1+\log(e+A/x))$ is increasing on $(0,\infty)$
for every $A>0$.

Consequently, we get
\[
\|u^\ep-u^0\|_{L_T^\infty(B_{1,1}^{0})}
\le C_{M,T}\ep\bigl(1+\log(e/\ep)\bigr).
\]

If $d\ge2$, choose $\beta=1/2$.
Arguing as in the proof of Theorem~\ref{thm:visc},
with \eqref{eq:p1-positive-multiplier} in place of
\eqref{eq:critical-multiplier}, we obtain
\[
\sum_{i=2}^{5}\|F_i\|_{B_{1,1}^{\beta}}
\le C_M\|u^\ep-u^0\|_{B_{1,1}^{\beta}}.
\]
Since $2+\beta\le1+d$, \eqref{eq:original-uniform-local} gives
\[
\begin{aligned}
\|F_1\|_{L_T^1(B_{1,1}^{\beta})}
&\le C\ep\|u^\ep\|_{L_T^1(B_{1,1}^{2+\beta})}\\
&\le C\ep T\|u^\ep\|_{L_T^\infty(B_{1,1}^{1+d})}
\le C_{M,T}\ep.
\end{aligned}
\]
Moreover, as above, we have
\[
\|\nabla\bigl((1-2u^0)\nabla S^0\bigr)\|
_{L_T^\infty(B_{1,1}^{d})}\le C_M.
\]
Applying Lemma~\ref{lem:transport-estimate} with
$p=p_1=r=1$ and $\sigma=\beta$ to \eqref{eq:error-equation},
we obtain, for $0\le t\le T$,
\[
\begin{aligned}
\|u^\ep-u^0\|_{L_t^\infty(B_{1,1}^{\beta})}
&\le C_{M,T}\sum_{i=1}^{5}
\|F_i\|_{L_t^1(B_{1,1}^{\beta})}\\
&\le C_{M,T}\ep
+C_{M,T}\int_0^t
\|u^\ep-u^0\|_{L_\tau^\infty(B_{1,1}^{\beta})}\,d\tau.
\end{aligned}
\]
Gronwall's inequality and the embedding
$B_{1,1}^{\beta}\hookrightarrow B_{1,1}^{0}$ yield
\[
\|u^\ep-u^0\|_{L_T^\infty(B_{1,1}^{0})}
\le C\|u^\ep-u^0\|_{L_T^\infty(B_{1,1}^{\beta})}
\le C_{M,T}\ep.
\]
This completes the proof.
\end{proof}

\section[Inviscid limit for p infinity]{Inviscid limit for $p=\infty$}
\label{sec:endpoint-vanishing-viscosity}

In this section, we prove Theorem~\ref{thm:c5-endpoint-vv} and study
the convergence rates and their optimality. Let $T=T_*$ and
$M=4\Bigl(1+\|u_0\|_{B_{\infty,1}^{1}}\Bigr)$.
Let $C_{M,T}=C(d,M,T)>0$ be independent
of $\ep$ and $N$. Additional dependencies will be indicated explicitly.
By Lemma~\ref{lem:uniform-local}, we have on $[0,T]$
\begin{equation}\label{eq:c5-uniform-endpoint}
\|u^\ep\|_{B_{\infty,1}^{1}}
+\|u^0\|_{B_{\infty,1}^{1}}
\le M.
\end{equation}
By the definition of $S_N$ and Lemma~\ref{lem:bernstein},
for $N\ge1$ and $s>1$, we have
\[
\|S_Nu_0\|_{B_{\infty,1}^{s}}
\le C_s2^{(s-1)N}\|u_0\|_{B_{\infty,1}^{1}},
\qquad
\|u_0-S_Nu_0\|_{L^\infty}
\le C2^{-N}\|u_0-S_Nu_0\|_{B_{\infty,1}^{1}}.
\]
Moreover, the $\ell^1$ summability defining
$B_{\infty,1}^{1}$ implies that
\[
\|u_0-S_Nu_0\|_{B_{\infty,1}^{1}}
\longrightarrow0,\qquad N\to\infty.
\]

\begin{lem}\label{lem:c5-uniform-difference}
Fix $\ep\in[0,1]$. Let $u$ and $v$ solve \eqref{eq:c5-unified}
with initial data $u_0$ and $v_0$, respectively. Assume that
\[
u,v\in C([0,T],B_{\infty,1}^{1}(\R^d)),
\qquad u(0)=u_0,\quad v(0)=v_0.
\]
Assume
\[
\|u\|_{L_T^\infty(B_{\infty,1}^{1})}+
\|v\|_{L_T^\infty(B_{\infty,1}^{1})}\le M.
\]
Then
\begin{equation}\label{eq:c5-difference-Linf}
\|u-v\|_{L_T^\infty(L^\infty)}\le C_{M,T}\|u_0-v_0\|_{L^\infty}.
\end{equation}
If, in addition, $v\in L^\infty(0,T;B_{\infty,1}^{2}(\R^d))$, then
\begin{equation}\label{eq:c5-difference-B1}
\|u-v\|_{L_T^\infty(B_{\infty,1}^{1})}
\le C_{M,T}\left(
\|u_0-v_0\|_{B_{\infty,1}^{1}}+
\int_0^T\|v(t)\|_{B_{\infty,1}^{2}}\|u(t)-v(t)\|_{L^\infty}\,dt
\right).
\end{equation}
The constants are independent of $\ep$.
\end{lem}

\begin{proof}
Subtracting \eqref{eq:c5-unified} for $v$ from that for $u$,
we obtain
\begin{equation}\label{eq:c5-difference-equation}
\begin{aligned}
 &\partial_t(u-v)-\ep\Delta(u-v)
 +(1-2u)\nabla(I-\Delta)^{-1}u\cdot\nabla(u-v)\\
 =&-\bigl((1-2u)\nabla(I-\Delta)^{-1}u
          -(1-2v)\nabla(I-\Delta)^{-1}v\bigr)\cdot\nabla v\\
 &-u(1-u)\Delta(I-\Delta)^{-1}u
       +v(1-v)\Delta(I-\Delta)^{-1}v.
\end{aligned}
\end{equation}
The identities
\[
\begin{aligned}
 &(1-2u)\nabla(I-\Delta)^{-1}u
       -(1-2v)\nabla(I-\Delta)^{-1}v\\
 =&(1-2u)\nabla(I-\Delta)^{-1}(u-v)
                   -2(u-v)\nabla(I-\Delta)^{-1}v
\end{aligned}
\]
and
\[
\begin{aligned}
 &u(1-u)\Delta(I-\Delta)^{-1}u
             -v(1-v)\Delta(I-\Delta)^{-1}v\\
 =&(u-v)(1-u-v)\Delta(I-\Delta)^{-1}u
                    +v(1-v)\Delta(I-\Delta)^{-1}(u-v)
\end{aligned}
\]
show, by \eqref{eq:c5-Linf-tools} and the assumption
\[
 \|u\|_{L_T^\infty(B_{\infty,1}^{1})}
 +\|v\|_{L_T^\infty(B_{\infty,1}^{1})}\le M,
\]
that the right-hand side of
\eqref{eq:c5-difference-equation} is bounded in $L^\infty$ by
$C_M\|u-v\|_{L^\infty}$.
Lemma~\ref{lem:scalar-estimate} with $p=\infty$ gives
\[
 \|u(t)-v(t)\|_{L^\infty}
 \le \|u_0-v_0\|_{L^\infty}
 +C_M\int_0^t
 \|u(\tau)-v(\tau)\|_{L^\infty}\,d\tau.
\]
Gronwall's inequality then yields
\[
 \|u(t)-v(t)\|_{L^\infty}
 \le e^{C_Mt}\|u_0-v_0\|_{L^\infty},
\]
which proves \eqref{eq:c5-difference-Linf}.

To prove \eqref{eq:c5-difference-B1}, we use
\eqref{eq:c5-tame-tools} and \eqref{eq:c5-multiplier-tools}
to obtain
\[
\begin{aligned}
 &\|\bigl((1-2u)\nabla(I-\Delta)^{-1}u
          -(1-2v)\nabla(I-\Delta)^{-1}v\bigr)
                    \cdot\nabla v\|_{B_{\infty,1}^{1}}\\
 &\quad\le C_M\|u-v\|_{B_{\infty,1}^{1}}
       +C_M\|v\|_{B_{\infty,1}^{2}}\|u-v\|_{L^\infty}.
\end{aligned}
\]
The difference of the two reaction terms has
$B_{\infty,1}^{1}$ norm at most $C_M\|u-v\|_{B_{\infty,1}^{1}}$. The transport velocity obeys
\[
\begin{aligned}
\|(1-2u)\nabla(I-\Delta)^{-1}u\|_{B_{\infty,1}^{1}}+\|\nabla\bigl((1-2u)\nabla(I-\Delta)^{-1}u\bigr)\|_{L^\infty}\le C_M.
\end{aligned}
\]
Applying \eqref{eq:c5-critical-transport-tools} to \eqref{eq:c5-difference-equation}  therefore yields
\begin{equation}\label{eq:c5-F-B1}
\begin{aligned}
 \|u(t)-v(t)\|_{B_{\infty,1}^{1}}
 &\le\|u_0-v_0\|_{B_{\infty,1}^{1}}
        +C_M\int_0^t\|u(\tau)-v(\tau)\|_{B_{\infty,1}^{1}}\,d\tau\\
 &\quad+C_M\int_0^t\|v(\tau)\|_{B_{\infty,1}^{2}}
                     \|u(\tau)-v(\tau)\|_{L^\infty}\,d\tau.
\end{aligned}
\end{equation}
Gronwall's inequality proves \eqref{eq:c5-difference-B1}.
\end{proof}

\begin{lem}\label{lem:c5-high-regularity}
Let $u$ be a smooth solution with diffusion coefficient $\ep\in[0,1]$ on an interval $[0,t_*]$, with $t_*\le T$ and $\|u\|_{L_{t_*}^\infty(B_{\infty,1}^{1})}\le K$. If $u(0)\in B_{\infty,1}^{s}$ for $s>1$, then
\begin{equation}\label{eq:c5-high-regularity-bound}
 \|u\|_{L_{t_*}^\infty(B_{\infty,1}^{s})}
 \le C_{s,K,T}\|u(0)\|_{B_{\infty,1}^{s}}.
\end{equation}
\end{lem}

\begin{proof}
Applying $\Delta_j$ to \eqref{eq:c5-unified} gives
\[
\begin{aligned}
 &\partial_t\Delta_ju+(1-2u)\nabla(I-\Delta)^{-1}u
                \cdot\nabla\Delta_ju-\ep\Delta\Delta_ju\\
= &-\Delta_j\bigl(u(1-u)\Delta(I-\Delta)^{-1}u\bigr)
   +[(1-2u)\nabla(I-\Delta)^{-1}u\cdot\nabla,\Delta_j]u.
\end{aligned}
\]
By Lemma~\ref{lem:scalar-estimate} with $p=\infty$ and
\eqref{eq:c5-commutator-tools}, together with
\eqref{eq:c5-tame-tools}, \eqref{eq:c5-multiplier-tools}
and \eqref{eq:c5-Linf-tools}, we obtain
\[
\begin{alignedat}{2}
 \|u(t)\|_{B_{\infty,1}^{s}}
 &\le\|u(0)\|_{B_{\infty,1}^{s}}
 &&{}+C_s\int_0^t
       \|\nabla\bigl((1-2u)\nabla(I-\Delta)^{-1}u\bigr)\|_{L^\infty}
       \|u\|_{B_{\infty,1}^{s}}\,d\tau\\
 &&&{}+C_s\int_0^t
       \|\nabla u\|_{L^\infty}
       \|(1-2u)\nabla(I-\Delta)^{-1}u\|_{B_{\infty,1}^{s}}\,d\tau\\
 &&&{}+C_s\int_0^t
       \|u(1-u)\Delta(I-\Delta)^{-1}u\|_{B_{\infty,1}^{s}}\,d\tau\\
 &\le\|u(0)\|_{B_{\infty,1}^{s}}
 &&{}+C_s\int_0^t
       \|\nabla\bigl((1-2u)\nabla(I-\Delta)^{-1}u\bigr)\|_{L^\infty}
       \|u\|_{B_{\infty,1}^{s}}\,d\tau\\
 &&&{}+C_s\int_0^t
       \|\nabla u\|_{L^\infty}
       \bigl(1+\|u\|_{L^\infty}\bigr)
       \|u\|_{B_{\infty,1}^{s}}\,d\tau\\
 &&&{}+C_s\int_0^t
       \|u\|_{L^\infty}
       \bigl(1+\|u\|_{L^\infty}\bigr)
       \|u\|_{B_{\infty,1}^{s}}\,d\tau\\
 &\le\|u(0)\|_{B_{\infty,1}^{s}}
 &&{}+C_{s,K}\int_0^t
       \|u(\tau)\|_{B_{\infty,1}^{s}}\,d\tau,
\end{alignedat}
\]
where the last inequality follows from
$\|u\|_{L_{t_*}^\infty(B_{\infty,1}^{1})}\le K$
and the embedding
$B_{\infty,1}^{1}\hookrightarrow W^{1,\infty}$.
Gronwall's inequality proves
\eqref{eq:c5-high-regularity-bound}.
\end{proof}

\begin{lem}\label{lem:regularization}
There exists $N_0\ge1$, depending on $u_0$ and $T$ but independent of
$\ep$, such that the solutions $u_N^\ep$ of \eqref{eq:c5-unified}
with initial data $S_Nu_0$ exist on $[0,T]$ for all $N\ge N_0$ and
$\ep\in[0,1]$. They satisfy
\begin{equation}\label{eq:c5-truncated-low-order}
 \|u_N^\ep\|_{B_{\infty,1}^{1}}
 +\|u_N^0\|_{B_{\infty,1}^{1}}
 \le2(M+1).
\end{equation}
For $s>1$,
\begin{equation}\label{eq:c5-high-regularity}
\|u_N^\ep\|_{L_T^\infty(B_{\infty,1}^{s})}
+\|u_N^0\|_{L_T^\infty(B_{\infty,1}^{s})}
\le C_{s,M,T}\|S_Nu_0\|_{B_{\infty,1}^{s}}.
\end{equation}
Moreover, for every $\ep\in[0,1]$,
\begin{equation}\label{eq:regularization-errors}
\begin{aligned}
\|u^\ep-u_N^\ep\|_{L_T^\infty(L^\infty)}
&\le C_{M,T}2^{-N}\|u_0-S_Nu_0\|_{B_{\infty,1}^{1}},\\
\|u^\ep-u_N^\ep\|_{L_T^\infty(B_{\infty,1}^{1})}
&\le C_{M,T}\|u_0-S_Nu_0\|_{B_{\infty,1}^{1}}.
\end{aligned}
\end{equation}
\end{lem}

\begin{proof}
The datum $S_Nu_0$ belongs to $B_{\infty,1}^{m}$ for every integer
$m\ge1$. Lemma~\ref{lem:uniform-local} gives a local critical solution. For $\ep=0$, local well-posedness in higher Besov spaces
follows from \cite{ZhouZhangMu2021}.
For each fixed $\ep>0$, the local construction in the proof of
Lemma~\ref{lem:uniform-local}, carried out in higher Besov spaces,
yields a local smooth solution. Uniqueness identifies these
solutions with the critical solution; denote it by $u_N^\ep$.
Whenever its critical norm stays bounded, Lemma~\ref{lem:c5-high-regularity}
bounds every higher norm. Restarting the local theory therefore
continues the smooth solution.

Recall that $M=4\bigl(1+\|u_0\|_{B_{\infty,1}^{1}}\bigr)$.
On any interval $[0,t]$ on which
\begin{equation}\label{eq:c5-bootstrap-critical}
 \|u_N^\ep\|_{L_t^\infty(B_{\infty,1}^{1})}\le M+1,
\end{equation}
we apply \eqref{eq:c5-high-regularity-bound} to $u_N^\ep$
with $s=2$ and $K=M+1$, and use
Lemma~\ref{lem:bernstein} to obtain
\[
\begin{aligned}
 \|u_N^\ep\|_{L_t^\infty(B_{\infty,1}^{2})}
 &\le C_{M,T}\|S_Nu_0\|_{B_{\infty,1}^{2}}\\
 &\le C_{M,T}2^N\|u_0\|_{B_{\infty,1}^{1}}\\
 &\le C_{M,T}2^N.
\end{aligned}
\]
By \eqref{eq:c5-uniform-endpoint} and
\eqref{eq:c5-bootstrap-critical}, we have
\[
 \|u^\ep\|_{L_t^\infty(B_{\infty,1}^{1})}
 +\|u_N^\ep\|_{L_t^\infty(B_{\infty,1}^{1})}
 \le M+(M+1)=2M+1.
\]
Applying Lemma~\ref{lem:c5-uniform-difference}
to $u^\ep$ and $u_N^\ep$ on $[0,t]$, we obtain
\[
\begin{aligned}
 \|u^\ep-u_N^\ep\|_{L_t^\infty(L^\infty)}
 &\le C_{M,T}\|u_0-S_Nu_0\|_{L^\infty}\\
 &\le C_{M,T}2^{-N}
       \|u_0-S_Nu_0\|_{B_{\infty,1}^{1}},
\end{aligned}
\]
and
\[
\begin{aligned}
 \|u^\ep-u_N^\ep\|_{L_t^\infty(B_{\infty,1}^{1})}
 &\le C_{M,T}\left(
       \|u_0-S_Nu_0\|_{B_{\infty,1}^{1}}
       +\int_0^t
        \|u_N^\ep(\tau)\|_{B_{\infty,1}^{2}}
        \|u^\ep(\tau)-u_N^\ep(\tau)\|_{L^\infty}
        \,d\tau\right)\\
 &\le C_{M,T}
       \|u_0-S_Nu_0\|_{B_{\infty,1}^{1}}.
\end{aligned}
\]
Since $\|u_0-S_Nu_0\|_{B_{\infty,1}^{1}}\to0$,
choose $N_0$ such that for all $N\ge N_0$
\[
  \|u^\ep-u_N^\ep\|_{L_t^\infty(B_{\infty,1}^{1})}\le C_{M,T}\|u_0-S_Nu_0\|_{B_{\infty,1}^{1}}
 \le\frac12.
\]
Fix $N\ge N_0$ and $\ep\in[0,1]$.
Then we get
\[
 \|u_N^\ep\|_{L_t^\infty(B_{\infty,1}^{1})}
 \le M+C_{M,T}\|u_0-S_Nu_0\|_{B_{\infty,1}^{1}}
 \le M+\frac12<M+1.
\]
The bootstrap and continuation argument therefore extends
$u_N^\ep$ to $[0,T]$, proving \eqref{eq:c5-truncated-low-order}
and \eqref{eq:regularization-errors}.
Applying \eqref{eq:c5-high-regularity-bound} with $K=M+1$
on $[0,T]$ gives \eqref{eq:c5-high-regularity}.
\end{proof}

\begin{proof}[Proof of Theorem~\ref{thm:c5-endpoint-vv}]
Using Lemma~\ref{lem:transport-diffusion}, we first establish
the uniform smoothing bound
\[
 \ep\|u^\ep\|_{L_T^1(B_{\infty,1}^{3})}\le C_{M,T}.
\]
By \eqref{eq:c5-unified}, \eqref{eq:c5-tame-tools}
and \eqref{eq:c5-multiplier-tools}, we have
\[
 \|(1-2u^\ep)\nabla(I-\Delta)^{-1}u^\ep\|_{B_{\infty,1}^{1}}
 +\|-u^\ep(1-u^\ep)\Delta(I-\Delta)^{-1}u^\ep\|_{B_{\infty,1}^{1}}\le C_M.
\]
\eqref{eq:prelim-derivative} gives
\[
 \|\nabla \bigl((1-2u^\ep)\nabla(I-\Delta)^{-1}u^\ep\bigr)\|_{B_{\infty,1}^{0}}
 \le C\|(1-2u^\ep)\nabla(I-\Delta)^{-1}u^\ep\|_{B_{\infty,1}^{1}}\le C_M.
\]
Applying Lemma~\ref{lem:transport-diffusion} to $u_N^\ep$
with $p=\infty$, $r=1$, $s=1$, and $\rho=\rho_1=1$,
and using Lemma~\ref{lem:regularization} and Fatou's lemma,
we obtain
\[
\ep\|u^\ep\|_{L_T^1(B_{\infty,1}^{3})}
\le \liminf_{N\to\infty}
\ep\|u_N^\ep\|_{L_T^1(B_{\infty,1}^{3})}
\le C_{M,T},
\qquad 0<\ep\le1.
\]
By \eqref{eq:c5-Linf-tools}, \eqref{eq:error-sources},
and \eqref{eq:c5-uniform-endpoint}, we have
\[
\begin{aligned}
\sum_{i=2}^{5}\|F_i\|_{L^\infty}
&\le C_M\bigl(
\|\nabla(S^\ep-S^0)\|_{L^\infty}
+\|u^\ep-u^0\|_{L^\infty}
+\|\Delta(S^\ep-S^0)\|_{L^\infty}
\bigr)\\
&\le C_M\|u^\ep-u^0\|_{L^\infty}.
\end{aligned}
\]
Applying Lemma~\ref{lem:scalar-estimate} with $p=\infty$
to \eqref{eq:error-equation} and using
$u^\ep(0)=u^0(0)$, we obtain
\[
\|u^\ep(t)-u^0(t)\|_{L^\infty}
\le \ep\int_0^t\|\Delta u^\ep(\tau)\|_{L^\infty}\,d\tau
+C_M\int_0^t\|u^\ep(\tau)-u^0(\tau)\|_{L^\infty}\,d\tau.
\]
Gronwall's inequality and Lemma~\ref{lem:bernstein} yield
\begin{equation}\label{eq:c5-direct-Linf-pre}
\|u^\ep-u^0\|_{L_T^\infty(L^\infty)}
\le C_{M,T}\ep\|\Delta u^\ep\|_{L_T^1(L^\infty)}
\le C_{M,T}\ep\|u^\ep\|_{L_T^1(B_{\infty,1}^{2})}.
\end{equation}
Taking $r=1$, $s_1=1$, $s_2=3$, and $\theta=1/2$
in \eqref{eq:besov-interpolation}, we obtain, for all $t\in[0,T]$,
\[
\|u^\ep(t)\|_{B_{\infty,1}^{2}}
\le C\|u^\ep(t)\|_{B_{\infty,1}^{1}}^{1/2}
\|u^\ep(t)\|_{B_{\infty,1}^{3}}^{1/2}.
\]
The smoothing bound just proved and H\"older's inequality in time yield
\[
\begin{aligned}
\ep\|u^\ep\|_{L_T^1(B_{\infty,1}^{2})}
&\le C\ep M^{1/2}T^{1/2}
\|u^\ep\|_{L_T^1(B_{\infty,1}^{3})}^{1/2}\\
&\le C_{M,T}\ep^{1/2}.
\end{aligned}
\]
Together with \eqref{eq:c5-direct-Linf-pre}, this proves \eqref{eq:c5-rate-Linf}.
By Lemma~\ref{lem:log}, \eqref{eq:c5-uniform-endpoint}, and
\eqref{eq:c5-rate-Linf}, we obtain
\[
\begin{aligned}
\|u^\ep-u^0\|_{L_T^\infty(B_{\infty,1}^{0})}
&\le C\|u^\ep-u^0\|_{L_T^\infty(L^\infty)}
\left(
1+\log\left(
e+\frac{C_M}{\|u^\ep-u^0\|_{L_T^\infty(L^\infty)}}
\right)
\right)\\
&\le C_{M,T}\sqrt\ep
\left(1+\log\left(e+\frac{C_M}{\sqrt\ep}\right)\right)\\
&\le C_{M,T}\sqrt\ep\log(e/\ep).
\end{aligned}
\]
This proves \eqref{eq:c5-rate-B0}.

For $N\ge N_0$, Lemma~\ref{lem:regularization} gives solutions
$u_N^\ep$ and $u_N^0$ to \eqref{eq:PKS} and \eqref{eq:HKSE},
respectively, on $[0,T]$, with common initial data $S_Nu_0$ and
\[
S_N^\ep=(I-\Delta)^{-1}u_N^\ep,
\qquad
S_N^0=(I-\Delta)^{-1}u_N^0.
\]
By \eqref{eq:error-equation}, we have
\begin{equation}\label{eq:c5-smooth-error}
\begin{cases}
\displaystyle
\partial_t(u_N^\ep-u_N^0)
+(1-2u_N^0)\nabla S_N^0\cdot\nabla(u_N^\ep-u_N^0)
=\sum_{i=1}^5F_{i,N},\\[1ex]
(u_N^\ep-u_N^0)(0)=0,
\end{cases}
\end{equation}
where $F_{i,N}$ are obtained from $F_i$ in
\eqref{eq:error-sources} by replacing $u^\ep$ and $u^0$
with $u_N^\ep$ and $u_N^0$, respectively.

By \eqref{eq:c5-high-regularity} with $s=2$ and
Lemma~\ref{lem:bernstein}, we obtain
\begin{equation*}
\|u_N^\ep\|_{L_T^\infty(B_{\infty,1}^{2})}
+\|u_N^0\|_{L_T^\infty(B_{\infty,1}^{2})}
\le C_{M,T}2^N.
\end{equation*}
Repeating the $L^\infty$ estimates for \eqref{eq:error-equation}
and using the regularity bounds for $u_N^\ep$, we obtain
\begin{equation}\label{eq:c5-wN-Linf}
\|u_N^\ep-u_N^0\|_{L_T^\infty(L^\infty)}
\le C_{M,T}\ep2^N.
\end{equation}
Similarly, taking $s=3$ in \eqref{eq:c5-high-regularity},
we obtain
\begin{equation*}
\|u_N^\ep\|_{L_T^\infty(B_{\infty,1}^{3})}
+\|u_N^0\|_{L_T^\infty(B_{\infty,1}^{3})}
\le C_{M,T}2^{2N}.
\end{equation*}
Using \eqref{eq:c5-critical-transport-tools},
Corollary~\ref{lem:bony}, Lemma~\ref{lem:elliptic},
and the bounds above, we obtain
\begin{equation}\label{eq:c5-wN-transport}
\begin{aligned}
\|u_N^\ep(t)-u_N^0(t)\|_{B_{\infty,1}^{1}}
&\le C_{M,T}\int_0^t\sum_{i=1}^{5}
\|F_{i,N}(\tau)\|_{B_{\infty,1}^{1}}\,d\tau\\
&\le C_{M,T}\ep2^{2N}
+C_{M,T}\int_0^t
\|u_N^\ep(\tau)-u_N^0(\tau)\|_{B_{\infty,1}^{1}}\,d\tau\\
&\quad+C_{M,T}2^N\int_0^t
\|u_N^\ep(\tau)-u_N^0(\tau)\|_{L^\infty}\,d\tau\\
&\le C_{M,T}\ep2^{2N}
+C_{M,T}\int_0^t
\|u_N^\ep(\tau)-u_N^0(\tau)\|_{B_{\infty,1}^{1}}\,d\tau,
\end{aligned}
\end{equation}
where the last inequality follows from \eqref{eq:c5-wN-Linf}.
Gronwall's inequality then yields
\begin{equation}\label{eq:c5-wN-B1}
\|u_N^\ep-u_N^0\|_{L_T^\infty(B_{\infty,1}^{1})}
\le C_{M,T}\ep2^{2N}.
\end{equation}
Finally, Lemma~\ref{lem:regularization} gives, for every $N\ge N_0$,
\begin{equation*}
 \|u^\ep-u_N^\ep\|_{L_T^\infty(L^\infty)}
 +\|u^0-u_N^0\|_{L_T^\infty(L^\infty)}
 \le C_{M,T}2^{-N}\|u_0-S_Nu_0\|_{B_{\infty,1}^{1}},
\end{equation*}
and
\[
 \|u^\ep-u_N^\ep\|_{L_T^\infty(B_{\infty,1}^{1})}
 +\|u^0-u_N^0\|_{L_T^\infty(B_{\infty,1}^{1})}
 \le C_{M,T}\|u_0-S_Nu_0\|_{B_{\infty,1}^{1}}.
\]
Combining the above estimates with \eqref{eq:c5-wN-B1}, we obtain
\begin{equation}\label{eq:c5-final-BS}
\|u^\ep-u^0\|_{L_T^\infty(B_{\infty,1}^{1})}
\le C_{M,T}\|u_0-S_Nu_0\|_{B_{\infty,1}^{1}}+C_{M,T}\ep2^{2N}.
\end{equation}
For sufficiently small $\ep$, set
\[
 N(\ep)=\left\lfloor\frac14\log_2\frac1\ep\right\rfloor
 \ge N_0.
\]
Then $N(\ep)\to\infty$ and $\ep2^{2N(\ep)}\lesssim\ep^{1/2}\to0$. Since $\|u_0-S_Nu_0\|_{B_{\infty,1}^{1}}\to0$,
\eqref{eq:c5-final-BS} proves \eqref{eq:c5-strong-B1}.
Combining \eqref{eq:c5-rate-Linf} with
Proposition~\ref{prop:c5-lower-fixed} proves
\eqref{eq:c5-sharp-two-sided} and the sharpness assertion.
\end{proof}

\begin{prop}\label{prop:higher-rates}
Under the hypotheses of Theorem~\ref{thm:c5-endpoint-vv}, assume also
$u_0\in B_{\infty,1}^{1+\sigma}$ with $\sigma>0$. Then
\[
 \|u^\ep-u^0\|_{L_T^\infty(B_{\infty,1}^{1})}
 \le C_{\sigma,u_0,T}\ep^{\min\{\sigma/2,1\}},\qquad
 \|u^\ep-u^0\|_{L_T^\infty(L^\infty)}
 \le C_{\sigma,u_0,T}\ep^{\min\{(1+\sigma)/2,1\}}.
\]
The constants are independent of $\ep$.
\end{prop}

\begin{proof}
By the Littlewood--Paley decomposition, we have
\begin{equation}\label{eq:c5-regular-tail}
\|u_0-S_Nu_0\|_{B_{\infty,1}^{1}}\le C_\sigma \|u_0\|_{B_{\infty,1}^{1+\sigma}}2^{-\sigma N},\qquad
\|(I-S_N)u_0\|_{L^\infty}\le C_\sigma \|u_0\|_{B_{\infty,1}^{1+\sigma}}2^{-(1+\sigma)N},
\end{equation}
and
\begin{equation*}
\|S_Nu_0\|_{B_{\infty,1}^{k}}\le C_\sigma \|u_0\|_{B_{\infty,1}^{1+\sigma}}
\left(1+2^{(k-1-\sigma)N}\right),\qquad k=2,3.
\end{equation*}
Lemma~\ref{lem:c5-high-regularity} thus yields, for $N\ge N_0$,
\[
\|u_N^\ep\|_{L_T^\infty(B_{\infty,1}^{k})}
+\|u_N^0\|_{L_T^\infty(B_{\infty,1}^{k})}
\le C_{\sigma,M,T}\|u_0\|_{B_{\infty,1}^{1+\sigma}}
\left(1+2^{(k-1-\sigma)N}\right),
\qquad k=2,3.
\]

Arguing as in the derivation of \eqref{eq:c5-wN-Linf}
and using the improved $B_{\infty,1}^{2}$ bound, we obtain
\begin{equation}\label{eq:c5-regular-smooth-Linf}
\|u_N^\ep-u_N^0\|_{L_T^\infty(L^\infty)}\le C_{\sigma,M,T}\ep \|u_0\|_{B_{\infty,1}^{1+\sigma}}\left(1+2^{(1-\sigma)N}\right).
\end{equation}
Using the preceding regularity bounds and
\eqref{eq:c5-regular-smooth-Linf} in the argument for
\eqref{eq:c5-wN-transport}, we obtain by Gronwall's inequality
\[
\begin{aligned}
\|u_N^\ep-u_N^0\|_{L_T^\infty(B_{\infty,1}^{1})}
&\le C_{\sigma,u_0,T}\ep
\left(1+2^{(2-\sigma)N}
+\left(1+2^{(1-\sigma)N}\right)^2\right)\\
&\le C_{\sigma,u_0,T}\ep
\left(1+2^{(2-\sigma)N}\right).
\end{aligned}
\]

Combining this estimate with Lemma~\ref{lem:regularization} and
\eqref{eq:c5-regular-tail} gives
\begin{equation}\label{eq:c5-regular-final-B1}
\|u^\ep-u^0\|_{L_T^\infty(B_{\infty,1}^{1})}
\le C_{\sigma,u_0,T}
\left(2^{-\sigma N}+\ep\left(1+2^{(2-\sigma)N}\right)\right).
\end{equation}
For $0<\sigma<2$ and sufficiently small $\ep$, choose
$N=\lfloor\tfrac12\log_2(1/\ep)\rfloor\ge N_0$.
For $\sigma\ge2$, take $N\ge N_0$ large enough such that
$2^{-\sigma N}\le\ep$ and apply \eqref{eq:c5-regular-final-B1}. The result is
\[
\|u^\ep-u^0\|_{L_T^\infty(B_{\infty,1}^{1})}
\le C_{\sigma,u_0,T}
\begin{cases}
\ep^{\sigma/2},&0<\sigma<2,\\[1mm]
\ep,&\sigma\ge2.
\end{cases}
\]
For the $L^\infty$ norm, Lemma~\ref{lem:c5-uniform-difference}, \eqref{eq:c5-regular-tail}, and \eqref{eq:c5-regular-smooth-Linf} give
\[
\|u^\ep-u^0\|_{L_T^\infty(L^\infty)}
\le C_{\sigma,u_0,T}
\left(2^{-(1+\sigma)N}+\ep\left(1+2^{(1-\sigma)N}\right)\right).
\]
Choose the same $N$ if $0<\sigma<1$, and let $N\to\infty$ if $\sigma\ge1$. Consequently,
\[
\|u^\ep-u^0\|_{L_T^\infty(L^\infty)}
\le C_{\sigma,u_0,T}
\begin{cases}
\ep^{(1+\sigma)/2},&0<\sigma<1,\\[1mm]
\ep,&\sigma\ge1.
\end{cases}
\]
For the remaining values of $\ep$, the estimates follow from the uniform bounds.
\end{proof}

Choose a bounded, nondecreasing function $F\ge0$, smooth away
from zero, such that
\[
F(r)=
\begin{cases}
0, & r=0,\\[1mm]
\displaystyle\frac{r}{[\log(e/r)]^2},
& 0<r\le\frac14,\\[3mm]
F(\frac12), & r\ge\frac12,
\end{cases}
\]
with a smooth nondecreasing continuation on $[\frac14,\frac12]$.
Let $f$ be the even $2\pi$-periodic extension of $\delta F(|x|)$
from $[-\pi,\pi]$, where $\delta>0$ is fixed and
$\delta\|{F}\|_{L^\infty}\le\frac14$. Define
\begin{equation}\label{eq:c5-lower-datum}
u_0(x)=\frac12+f(x_1),\qquad x\in\R^d.
\end{equation}

\begin{prop}\label{prop:c5-lower-fixed}
The datum \eqref{eq:c5-lower-datum} belongs to
$B_{\infty,1}^{1}(\R^d)$.
Let $u^\ep,u^0$ be the corresponding solutions on the common
interval $[0,T]$ given by Lemma~\ref{lem:uniform-local}.
For each fixed $t_*\in(0,T]$, there exist $c,\ep_*>0$,
independent of $\ep$, such that
\[
\|{u^\ep(t_*)-u^0(t_*)}\|_{L^\infty}
\ge \frac{c\sqrt{\ep}}{[\log(e/\ep)]^2},
\qquad 0<\ep\le\ep_*.
\]
Thus, for this fixed datum, no convergence rate $O(\ep^\alpha)$
with $\alpha>1/2$ holds in either $L_T^\infty(L^\infty)$
or $L_T^\infty(B_{\infty,1}^{0})$.
\end{prop}

\begin{proof}
Since $f\in C_b^1$ and
$\omega(f',h)\le C[\log(e/h)]^{-2}$ for small $h$,
Lemma~\ref{lem:dini} gives the asserted regularity.
Uniqueness preserves the one-dimensional, even and periodic
structure, and the common periodic mean is conserved.
We henceforth write $x=x_1$.

Since the periodic means of $u^\ep$ and $u^0$ coincide,
\[
\int_0^x (u^\ep-u^0)(t,y)\,dy
\]
is periodic and vanishes at $t=0$.
Integrating the conservative equations and using evenness yields
\[
\begin{aligned}
&\partial_t\int_0^x (u^\ep-u^0)(t,y)\,dy
-\ep\partial_x^2\int_0^x (u^\ep-u^0)(t,y)\,dy +(1-u^\ep-u^0)\partial_xS^\ep(u^\ep-u^0)\\
=&\ep\partial_xu^0
+u^0(1-u^0)\Biggl(
\int_0^x (u^\ep-u^0)(t,y)\,dy
-(I-\partial_x^2)^{-1}
\int_0^x (u^\ep-u^0)(t,y)\,dy
\Biggr).
\end{aligned}
\]
Moreover,
\[
S^\ep-S^0
=\partial_x(I-\partial_x^2)^{-1}
\int_0^x (u^\ep-u^0)(t,y)\,dy.
\]
The maximum principle, the elliptic $L^\infty$ estimates
and Gronwall's inequality therefore imply
\begin{equation}\label{eq:c5-lower-potential}
\Biggl\|
\int_0^x (u^\ep-u^0)(t,y)\,dy
\Biggr\|_{L_T^\infty(L^\infty)}
+\|S^\ep-S^0\|_{L_T^\infty(L^\infty)}
\le C\ep.
\end{equation}

Evaluating the hyperbolic equation at $x=0$ and subtracting
it from the parabolic equation, we obtain
\[
\begin{aligned}
&\partial_t\bigl(u^\ep-u^0(t,0)\bigr)
-\ep\partial_x^2\bigl(u^\ep-u^0(t,0)\bigr)+(1-2u^\ep)\partial_xS^\ep
          \partial_x\bigl(u^\ep-u^0(t,0)\bigr)\\
=&\Bigl[
(1-u^\ep-u^0(t,0))(u^\ep-S^\ep)
+u^0(t,0)\bigl(1-u^0(t,0)\bigr)
\Bigr]\bigl(u^\ep-u^0(t,0)\bigr)\\
& +u^0(t,0)\bigl(1-u^0(t,0)\bigr)
          \bigl(S^0(t,0)-S^\ep(t,x)\bigr).
\end{aligned}
\]
The critical bounds, evenness and
\eqref{eq:c5-lower-potential} allow us to choose $C\ge1$,
independent of $\ep$, such that
\[
\begin{aligned}
&\bigl|(1-2u^\ep)\partial_xS^\ep\bigr|\le C|x|,\\
&\Bigl|
(1-u^\ep-u^0(t,0))(u^\ep-S^\ep)
+u^0(t,0)\bigl(1-u^0(t,0)\bigr)
\Bigr|\le C,\\
&\Bigl|
u^0(t,0)\bigl(1-u^0(t,0)\bigr)
\bigl(S^0(t,0)-S^\ep(t,x)\bigr)
\Bigr|\le C(x^2+\ep).
\end{aligned}
\]
Increasing $C$ if necessary, we also assume that
$C\pi^2\ge\delta\|F\|_{L^\infty}$.

Consider the comparison function
\[
Q(t,x)
=e^{-Ct}
\Bigl[
e^{\frac{\ep}{2C}(1-e^{-2Ct})\partial_x^2}
\bigl(\delta F(|\cdot|)\bigr)
\Bigr](e^{-Ct}x)
-Ce^{(3+3C)t}(x^2+\ep).
\]
Here the heat semigroup acts on the nonperiodic function
$\delta F(|x|)$ on $\R$.
The heat term is nonnegative and nondecreasing in $|x|$.
Using these properties and the preceding bounds,
direct differentiation gives
\[
\begin{aligned}
&\partial_tQ-\ep\partial_x^2Q
 +(1-2u^\ep)\partial_xS^\ep\,\partial_xQ-\Bigl[
(1-u^\ep-u^0(t,0))(u^\ep-S^\ep)
+u^0(t,0)\bigl(1-u^0(t,0)\bigr)
\Bigr]Q\\
\le&
-Ce^{(3+3C)t}
\Bigl[(3+2C)(x^2+\ep)-2\ep-2Cx^2\Bigr]\\
\le& -Ce^{(3+3C)t}(x^2+\ep)\\
\le&
u^0(t,0)\bigl(1-u^0(t,0)\bigr)
\bigl(S^0(t,0)-S^\ep(t,x)\bigr).
\end{aligned}
\]
By the choice of $C$,
\[
Q(0,x)=\delta F(|x|)-C(x^2+\ep)\le f(x)
\qquad\text{for all }x\in\R.
\]
Moreover, $Q(t,x)\to-\infty$ as $|x|\to\infty$,
uniformly for $t\in[0,T]$.
Lemma~\ref{lem:comparison} therefore yields
\[
u^\ep(t,x)-u^0(t,0)\ge Q(t,x).
\]

Fix $t_*\in(0,T]$.
For $0<\ep\le1/64$, the Gaussian formula gives
\[
\begin{aligned}
&u^\ep(t_*,0)-u^0(t_*,0)\\
\ge&
\frac{\delta e^{-Ct_*}}
{\sqrt{\frac{2\pi\ep}{C}(1-e^{-2Ct_*})}}
\int_{\R}
\exp\Bigl(
-\frac{Cy^2}{2\ep(1-e^{-2Ct_*})}
\Bigr)F(|y|)\,dy
-Ce^{(3+3C)t_*}\ep\\
\ge&
\frac{\delta e^{-Ct_*}}
{\sqrt{\frac{2\pi\ep}{C}(1-e^{-2Ct_*})}}
\int_{\sqrt\ep\le|y|\le2\sqrt\ep}
\exp\Bigl(
-\frac{Cy^2}{2\ep(1-e^{-2Ct_*})}
\Bigr)
\frac{|y|}{[\log(e/|y|)]^2}\,dy
-Ce^{(3+3C)t_*}\ep\\
\ge&
\delta e^{-Ct_*}
\sqrt{\frac{2C}{\pi(1-e^{-2Ct_*})}}
\exp\Bigl(-\frac{2C}{1-e^{-2Ct_*}}\Bigr)
\frac{\sqrt\ep}{[\log(e/\ep)]^2}
-Ce^{(3+3C)t_*}\ep.
\end{aligned}
\]
Since $t_*>0$ is fixed and
$\sqrt\ep[\log(e/\ep)]^2\to0$ as $\ep\to0$,
the last estimate gives
\[
u^\ep(t_*,0)-u^0(t_*,0)
\ge \frac{c\sqrt\ep}{[\log(e/\ep)]^2}
\]
for all sufficiently small $\ep$.
The remaining assertions follow from this lower bound and
$B_{\infty,1}^{0}\hookrightarrow L^\infty$.
\end{proof}

\smallskip
\noindent\textbf{Acknowledgments}

\noindent
This work is supported by the National Natural Science Foundation of China
(No.~12571261).
	
	\smallskip
	\noindent\textbf{Conflict of interest}
	The authors declare that they have no conflict of interest.
	
	\smallskip
	\noindent\textbf{Data Availability}
	No data was used for the research described in the article.

\end{document}